\documentclass{amsart}
\usepackage{amsmath,amsthm,amssymb,amsfonts,enumerate,color}
\usepackage{mathrsfs}
\usepackage{amsmath,amstext,amsthm,amssymb,amsxtra}
\usepackage{txfonts} 
\usepackage[colorlinks, citecolor=blue,pagebackref,hypertexnames=false]{hyperref}
\allowdisplaybreaks
\usepackage{pgf,tikz}
\usepackage{relsize}

\begin{document}



	\newtheorem{theorem}{Theorem}[section]
	\newtheorem{proposition}[theorem]{Proposition}
	\newtheorem{corollary}[theorem]{Corollary}
	\newtheorem{lemma}[theorem]{Lemma}
	\newtheorem{definition}[theorem]{Definition}	
	\newtheorem{assum}{Assumption}[section]
	\newtheorem{example}[theorem]{Example}
	\newtheorem{remark}[theorem]{Remark}
	\newtheorem*{conjecture}{Conjecture}

\newcommand{\R}{\mathbb{R}}
\newcommand{\Rn}{{\mathbb{R}^n}}
\newcommand{\T}{\mathbb{T}}
\newcommand{\s}{s_{*}}
\newcommand{\sn}{s_{n,*}}
\newcommand{\Tm}{{\mathbb{T}^m}}
\newcommand{\TmRn}{{\mathbb{T}^m\times\mathbb{R}^n}}
\newcommand{\Z}{\mathbb{Z}}
\newcommand{\Zm}{{\mathbb{Z}^m}}
\newcommand{\C}{\mathbb{C}}
\newcommand{\Q}{\mathbb{Q}}
\newcommand{\N}{\mathbb{N}}

	\newcommand{\supp}{{\rm supp}{\hspace{.05cm}}}
	\newcommand {\lb}{{\langle}}
	\newcommand {\rb}{\rangle}
 	\numberwithin{equation}{section}

\title[$L^p\to L^q$ estimates for    Stein's   spherical  maximal    operators]{  $L^p\to L^q$ estimates for     Stein's    spherical  maximal    operators}

 \author[ N. Liu,  M. Shen, \ L. Song  and L.  Yan]{ Naijia Liu, \ Minxing Shen, \ Liang Song
	and \ Lixin Yan}

\address{Naijia Liu,
	Department of Mathematics,
	Sun Yat-sen University,
	Guangzhou, 510275,
	P.R.~China}
\email{liunj@mail2.sysu.edu.cn}	

\address{Minxing Shen, Chern Institute of Mathematics, Nankai University, Tianjin, 300071, P.R. China}
\email{9820240068@nankai.edu.cn}	

\address{Liang Song,
	Department of Mathematics,
	Sun Yat-sen University,
	Guangzhou, 510275,
	P.R.~China}
\email{songl@mail.sysu.edu.cn}

\address{
	Lixin Yan, Department of Mathematics, Sun Yat-sen  University, Guangzhou, 510275, P.R. China}
\email{mcsylx@mail.sysu.edu.cn}

\date{\today}
\subjclass[2020]{42B25, 42B20, 35L05}

\keywords{spherical  maximal   operators, $L^p$-improving estimates,  wave equation,  local smoothing estimates}

\begin{abstract} 
	In this article we	consider  a  modification     of the Stein's spherical    maximal  operator
		of  complex  order $\alpha$ on
		${\mathbb R^n}$:  
	$$
		{\mathfrak M}^\alpha_{[1,2]} f(x) =\sup\limits_{t\in [1,2]} \big|	{1\over   \Gamma(\alpha)  }  \int_{|y|\leq 1} \left(1-|y|^2 \right)^{\alpha -1} f(x-ty) dy\big|.
	$$
We show that when $n\geq 2$,  suppose
		$\|	{\mathfrak M}^{\alpha}_{[1,2]}  f \|_{L^q({\mathbb R^n})} \leq C\|f \|_{L^p({\mathbb R^n})}$
	holds for some $\alpha\in \C$, $p,q\geq1$,  then  we must  have that $q\geq p$ and
$${\rm Re}\,\alpha\geq  \sigma_n(p,q):=\max\left\{\frac{1}{p}-\frac{n}{q},\ \frac{n+1}{2p}-\frac{n-1}{2}\left(\frac{1}{q}+1\right),\frac{n}{p}-n+1\right\}.$$  Conversely,  we show that  ${\mathfrak M}^\alpha_{[1,2]}$  is bounded from $L^p({\mathbb R^n})$ to $L^q({\mathbb R^n})$ provided that 
   $q\geq p$ and ${\rm Re}\,\alpha>\sigma_2(p,q)$ for $n=2$; 
  and ${\rm Re}\,\alpha>\max\left\{\sigma_n(p,q), \
1/(2p)- (n-2)/(2q) -(n-1)/4\right\}$ 
 for $n>2$. The range of $\alpha,p$ and $q$  is almost optimal in the case either $n=2$, or  $\alpha=0$, or $(p,q)$ lies in some  regions for $n>2$.
	\end{abstract}

	\maketitle

\section{Introduction}
\setcounter{equation}{0}

 In 1976 Stein \cite{St}  introduced
 the spherical  maximal  means ${\frak M}^\alpha f(x)= \sup_{t>0}
 |{\frak M}^\alpha_tf (x) |$  of (complex) order $\alpha$ on ${\mathbb R}^n$, where
 \begin{eqnarray}\label{e1.1}
 	{\frak M}^\alpha_t f (x) =
 	{1\over   \Gamma(\alpha)  }  \int_{|y|\leq 1} \left(1-{|y|^2 }\right)^{\alpha -1} f(x-ty)\,\text{d}y.
 \end{eqnarray}
These means are   initially defined only for ${\rm Re}\  \alpha>0$, but the definition can be extended to   all complex $\alpha$  by analytic continuation. In the  case $\alpha=1$,  ${\frak M}^\alpha$ corresponds to the Hardy-Littlewood maximal operator and in the case  $\alpha=0$,  ${\frak M}^\alpha$  corresponds to   the spherical maximal operator  ${\frak M}  f(x):= \sup_{t>0}
|{\frak M}_tf (x) |$ in which
\begin{eqnarray}\label{e1.2}
	{\frak M}_tf (x) = c_n \int_{{\mathbb S}^{n-1}}   f(x-ty) \,\text{d}\sigma(y), \ \ \  x  \in {\mathbb R^n},
\end{eqnarray}
where  ${\mathbb S}^{n-1}$ denotes the standard unit sphere in ${\mathbb R^n}$.
In \cite{St}   Stein   obtained the  inequality
\begin{eqnarray}\label{e1.3}
	\|	{\frak M}^{\alpha}  f \|_{L^p({\mathbb R^n})} \leq C\|f \|_{L^p({\mathbb R^n})}
\end{eqnarray}
for $
	{\rm Re}\, \alpha>1-n+ {n/p} $ when $   1<p\leq 2$;
or
$ {\rm Re}\, \alpha>{(2-n)/p}  $ when $ 2\leq p\leq \infty.
$
From it, we see  that when $\alpha=0$ and $n\geq 3$,  the  maximal operator ${\frak M} $ is bounded on $L^p({\mathbb R^n})$ for the   range $ p>n/(n-1)$. This range of $p$ is sharp,
as has been pointed out in \cite{St, SWa},  no such result can hold for $p\leq n/(n-1)$
if $n\geq 2$. The extension of this result in \cite{St} to the case $n=2$ was established about a decade later by    Bourgain \cite{Bo}, see also the account in \cite[Chapter XI]{St1}.

In addition to Stein and Bourgain, other authors have studied the spherical  maximal  means; for instance see  \cite{Lee, LSSY, MSS, MYZ, NRS, Sc1, Sc, SS, S1, S2, Z}   and the references therein. All these refinements can be stated altogether as follows: When $n\geq 2$,  suppose \eqref{e1.3}
holds for some $\alpha$ and $p\geq 2$,  then  we must  have that
$
{\rm Re}\,\alpha \geq   \max \{1/p-(n-1)/2,\ -(n-1)/p \}.
$
Further,  estimate  \eqref{e1.3} holds
whenever    $p\geq 2$  and
\begin{eqnarray}\label{e1.7}
	{\rm Re} \, \alpha >
	\left\{
	\begin{array}{ll}
	 \max\left\{{ 1\over p} -{1\over 2}, \,  {1-n\over p} \right\}, \ \ \ &n=2;\\[4pt]
		 \max\left\{{ 1-n\over 4} +{3-n\over 2p}, \,  {1-n\over p} \right\}, \ \ \ &n\geq 3.
	\end{array}\right.
\end{eqnarray}

\smallskip

\subsection{Main results}
	In this article we	 modify the definition of the Stein's spherical    maximal  operator ${\frak M}^\alpha  $ so that the supremum
is taken over, say, $1\leq  t\leq 2$, i.e., 
$$
{\frak M}^\alpha_{[1,2]} f(x):=\sup\limits_{t\in [1,2]} \big|{\frak M}^\alpha_t f (x)\big|,
$$
 then the resulting   maximal function is
actually bounded from $L^p(\Rn)$ to $L^q(\Rn)$ for some $q>p$. 
More precisely, we have  the following results.

\begin{theorem}\label{th1.1}
Let  $n\geq 2$ and $p,q\in [1,\infty]$. Suppose
\begin{align}\label{e1.8}
\big\|{\frak M}^{\alpha}_{[1,2]}f\big\|_{L^{q}(\Rn)}\leq C\|f\|_{L^{p}(\Rn)}
\end{align}
holds  for some $\alpha\in\mathbb{C}$. Then  we must have $q\geq p$ and
		$$
		{\rm Re}\,\alpha\geq  \sigma_n(p,q):=\max\left\{\frac{1}{p}-\frac{n}{q},\ \frac{n+1}{2p}-\frac{n-1}{2}\left(\frac{1}{q}+1\right),\frac{n}{p}-n+1\right\}.
		$$
\end{theorem}
	
\begin{theorem}	\label{th1.2}
Assume $q\geq p$.
		
 {\rm (i)} \   Let  $n=2$. If  \ ${\rm Re}\,\alpha>\sigma_2(p,q)$,		then   the estimate \eqref{e1.8} holds.

{\rm (ii)} \  Let $n>2$. If	
$${\rm Re}\,\alpha>d_n(p,q):=\max\left\{\sigma_n(p,q), \
\frac{1}{2p}-\frac{n-2}{2q}-\frac{n-1}{4}\right\},
$$
then   the estimate \eqref{e1.8} holds.
\end{theorem}

From  Theorem \ref{th1.1} and  (i) of Theorem \ref{th1.2},  we see  that when  $n=2$, the range of $\alpha$ is  sharp  except for the boundary case  ${\rm Re}\ \alpha=\sigma_2(p,q)$ when $q\geq p$.

Note  that if $({1}/{p},{1}/{q})$ belongs to the set $\triangle ODE\backslash \triangle ABC$ (see Figure 1 below), then $\sigma_n(p,q)\geq
\frac{1}{2p}-\frac{n-2}{2q}-\frac{n-1}{4}$, and so $d_n(p,q)=\sigma_n(p,q)$.
\begin{figure}[h]
	\begin{center}
		\begin{tikzpicture}[scale=3.5]
			\draw[->] (0,0) -- (1.2,0) node[right] {\(\frac1p\)};
			\draw[->] (0,0) -- (0,1.05) node[above] {\(\frac1q\)};
			
			\draw[thick] (1,0) -- (1,1);
			\draw[thick] (0,0) -- (1,1);
			\draw[thick] (0.5,0.5) -- (0.46,0.39);
			\draw[thick] (0.46,0.39) -- (7/20,7/20);
			\draw[densely dotted] (0,1) -- (1,1);
			
			\node[below left] at (0,0) {\(O(0,0)\)};
			\node[left] at (7/20,7/20) {\(   A(\frac{n-1}{2(n+1)},\frac{n-1}{2(n+1)})\)};
			\node[below right] at (0.46,0.39) {\(B(\frac{(n-1)(n+3)}{2(n^{2}+2n-1)},\frac{(n-1)(n+1)}{2(n^{2}+2n-1)})\)};
			\node[above left] at (0.5,0.5) {\(C(\frac{1}{2},\frac{1}{2})\)};
			\node[above] at (1,1) {\(D(1,1)\)};
			\node[below] at (1,0) {\(E(1,0)\)};
			
			\fill[gray!50,opacity=0.6] (0.5,0.5) -- (0.46,0.39) --(7/20,7/20) -- cycle;
			
			\filldraw[black] (0.5,0.5) circle (0.2pt);
			\filldraw[black] (0.46,0.39) circle (0.2pt);
			\filldraw[black] (7/20,7/20) circle (0.2pt);
		\end{tikzpicture}
		\caption{  The range of $(1/p, 1/q)$  in (ii) of Theorem \ref{th1.2}.}
	\end{center}
\end{figure}
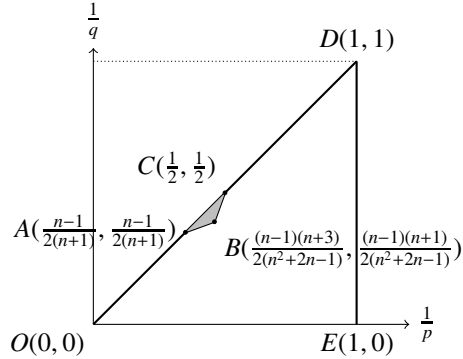

\noindent
 It then follows from  Theorem \ref{th1.1} and  (ii) of Theorem \ref{th1.2}   that if $n>2$, the range of $\alpha$ is   sharp when $({1}/{p},{1}/{q})$ belongs to the set $\triangle ODE\backslash \triangle ABC$ except for the boundary case  ${\rm Re}\ \alpha=\sigma_n(p,q)$.

When $n>2$ and $\alpha=0$, (ii) of Theorem \ref{th1.2} implies that the result of  Schlag and Sogge \cite{SS}, which is  optimal except the boundaries.
Indeed,  in the case $\alpha=0$  almost sharp results about the spherical maximal operators ${\frak M}_{[1,2]}$   have been obtained in Schlag \cite{Sc} for $n=2$,  and  Schlag and Sogge \cite{SS} for $n\geq 2$.    Recall that from  \cite{Sc,SS},   a necessary condition for $L^p(\R^{n})\to L^q(\R^{n})$ boundedness of ${\frak M}_{[1,2]}$   is  that $(1/p,1/q)$ belongs to the closed quadrangle $\mathcal{Q}$ with corners $P_1=(0,0)$, $P_2=(\frac{n-1}{n},\frac{n-1}{n})$, $P_3=(\frac{n-1}{n},\frac{1}{n})$ and $P_4=(\frac{n(n-1)}{n^2+1},\frac{n-1}{n^2+1})$ (when $n=2$, the quadrangle
${\mathcal Q}$ becomes a triangle as the points $P_2$ and $P_3$ coincide);    Further, if  $(1/p,1/q)$ belongs to the interior of $\mathcal Q$,  then ${\frak M}_{[1,2]}$ is bounded from $L^p(\R^{n})$ to $L^q(\R^{n})$.

\subsection{About our method}
Let us outline the proof of    Theorems~\ref{th1.1} and ~\ref{th1.2}.
To show Theorem~\ref{th1.1}, we use the asymptotic expansion of Fourier multiplier of the operator ${\frak M}^{\alpha}_t$ to see
  that ${\frak M}_t^\alpha $ are essentially the linear combination of half-wave operators
$e^{it\sqrt{-\Delta}}\lb D\rb^{-\frac{n-1}{2}-\alpha} $ and $e^{-it\sqrt{-\Delta}}\lb D\rb^{-\frac{n-1}{2}-\alpha}$, and hence the complexity  of the operator ${\frak M}^{\alpha}_t$
 comes from  the interference between the  operators $e^{it\sqrt{-\Delta}} $ and $e^{-it\sqrt{-\Delta}}$. To
show  the necessity of $L^p\to L^q$ boundedness of   ${\frak M}^{\alpha}_t$ in Theorem~\ref{th1.1},   we
construct three special examples  such that there is no interference between $e^{it\sqrt{-\Delta}}f$ and $e^{-it\sqrt{-\Delta}}f$.

 The proof  Theorem~\ref{th1.2} is shown   by  combining      $L^p\to L^q$ local smoothing estimates for wave operators,     and the techniques previously used
in \cite{MSS} and \cite{MYZ}. To obtain  $L^p\to L^q$ local smoothing estimates for  wave operators, we will apply $L^p\to L^p$ local smoothing estimates for wave operators and  interpolation.   We mention that local smoothing conjecture was originally formulated by Sogge \cite{S2}:
For $n\geq 2$ and $p\geq 2n/(n-1)$, one has
\begin{eqnarray*}
	\left \| u \right\|_{L^p({\mathbb R^n}\times [1,2])}  \leq C\left(\|f\|_{W^{\gamma, p}({\mathbb R^n})} + \|g\|_{W^{\gamma-1, p}({\mathbb R^n})} \right), \ \
	\ \ \ \ \
	{\rm if}\ \  \gamma>{n-1\over 2}-{n\over p},
\end{eqnarray*}
where
	\begin{eqnarray*}
		u(x,t)= \cos(t\sqrt{-\Delta}) f(x) +{\sin(t\sqrt{-\Delta}) \over \sqrt{-\Delta}} g(x).
\end{eqnarray*}
The local smoothing conjecture has been studied in numerous papers, see for instance \cite{BD, GS,  GWZ, LW, MYZ, MSS,   S1,    W} and the references therein. When $n=2$, sharp results follow by the work of Guth, Wang and Zhang \cite{GWZ}. When $n\geq 3$, the conjecture holds for all $p\geq {2(n+1)/ (n-1)}$ by the Bourgain-Demeter decoupling theorem \cite{BD} and the method of \cite{W}.
Up to now, the conjecture is still  open in the case  $2n/(n-1)\leq p<2(n+1)/(n-1)$ and $n>2$.

The paper is organized as follows. In Section 2, we give some basic results including the properties of
the Fourier multiplier associated to the    spherical   operators $	{\frak M}^\alpha_t $ by using asymptotic expansions of Bessel functions.
 The proof of  Theorem~\ref{th1.1} will be given in  Section 3 by constructing three examples to obtain the necessarity of $L^p \to L^q$ boundedness for the maximal operator $	{\frak M}^\alpha_{[1,2]}$.
The proof of Theorem~\ref{th1.2} will be  given in Section 4.

\section{Preliminary results}\label{s2}
\setcounter{equation}{0}

Recall that the  spherical   function is defined by  ${\frak M}^\alpha_t f(x)=f\ast m_{\alpha,t}(x)$
where $m_{\alpha,t}(x)= t^{-n}m_{\alpha}(t^{-1}x)$
and
$$
	m_{\alpha}(x)=\Gamma(\alpha)^{-1}\big(1-|x|^{2}\big)_{+}^{\alpha-1},
$$
where $\Gamma(\alpha)$ is the Gamma function and $(r)_+=\max\{0,r\}$ for $r\in \mathbb R$. Define the Fourier transform of $f$ by
$
\hat{f}(\xi)=\int_{\Rn} e^{2\pi ix \cdot\xi}f(x)\,\text{d}x.
$
It follows by    \cite[p.171]{SW}  that the Fourier transform of $m_\alpha$ is given by
\begin{align}\label{e2.1}
	\widehat{m_{\alpha}}(\xi)=\pi^{-\alpha+1}|\xi|^{-n/2-\alpha+1}
	J_{n/2+\alpha-1}\big(2\pi|\xi|\big).
\end{align}
Here $J_\beta$ denotes the Bessel function of order $\beta$.
For any complex number $\beta$, we can obtain the complete
 asymptotic expansion
\begin{align}\label{ebb}
	J_\beta(r)\sim  r^{-1/2}e^{ir} \sum_{j= 0}^{\infty}b_{j} r^{-j} + r^{-1/2}e^{-ir} \sum_{j= 0}^{\infty}d_j r^{-j},\, \ \ \ \ \ \  r\geq 1
\end{align}
for suitable coefficients $b_j$ and  $d_j$ with   $b_0,d_0\neq 0$.
Note that when  $\beta$ is a positive integer,  \eqref{ebb} is given in \cite[(15), p.338]{St1}. For general  $\beta$, we refer it to \cite[(1). 7.21, p.199]{Wa}.
Then there exists an error term $E(r)$ such that for any given $N\geq 1$ and  $r\geq 1$,
\begin{align}\label{e2.2}
	J_\beta(r)= r^{-1/2}e^{ir}\left(\sum_{j= 0}^{N-1}b_{j} r^{-j}+E_{N,1}(r)\right)+ r^{-1/2}e^{-ir}\left(\sum_{j= 0}^{N-1}d_j r^{-j}+ E_{N,2}(r)\right)+E(r),
\end{align}
where
$$
E_{N,1}(r) =\sum_{j= N}^{\infty}b_{j} r^{-j}, \ \ \ {\rm and} \ \ \
E_{N,2}(r) =\sum_{j= N}^{\infty}d_{j} r^{-j}
$$
satisfying
\begin{align}\label{e2.3}
	\bigg|\bigg(\frac{d}{dr}\bigg)^k E_{N,1}(r)\bigg|+\bigg|\bigg(\frac{d}{dr}\bigg)^k E_{N,2}(r)\bigg|+\bigg|\bigg(\frac{d}{dr}\bigg)^k E(r)\bigg|\leq C_{k}r^{-N-k}
\end{align}
for all $k\in \Z_+$.
We rewrite  \eqref{e2.1} as
\begin{align}\label{e2.4}
	\widehat{m_\alpha}(\xi)
	&= \varphi(|\xi|)\widehat{m_\alpha}(\xi)
	 + (1-\varphi(|\xi|))\widehat{m_\alpha}(\xi)   \notag\\
	&= \left[ \varphi(|\xi|)\widehat{m_\alpha}(\xi)+{\mathcal E}(|\xi|) \right ] \notag\\
	&+ \left[e^{2\pi i|\xi|}{\mathcal E}_{N,1}(|\xi|)+e^{-2\pi i|\xi|}{\mathcal E}_{N,2}(|\xi|) \right]\notag\notag\\
	& + |\xi|^{-(n-1)/2-\alpha}\left[e^{2\pi i|\xi|}a_1(|\xi|)+e^{-2\pi i|\xi|}a_2(|\xi|) \right],
\end{align}
where
\begin{align}\label{ee2.1}
&{\mathcal E}(r) = (2\pi)^{1/2}c(\pi, \alpha)(1-\varphi(r))r^{-(n-2)/2-\alpha}E(2\pi r),\nonumber\\
&{\mathcal E}_{N,\ell}(r) = c(\pi, \alpha) E_{N,\ell}(2\pi r)(1-\varphi(r))r^{-(n-1)/2-\alpha},\ \ \ \ell=1,2, \nonumber\\
&a_1(r) =  c(\pi, \alpha) \sum_{j= 0}^{N-1}b_j (2\pi r)^{-j}(1- \varphi(r)),\\
 &a_2(r) =   c(\pi, \alpha)\sum_{j= 0}^{N-1}d_j (2\pi r)^{-j}(1- \varphi(r))\nonumber
\end{align}
with $c(\pi, \alpha)= 2^{-1/2}\pi^{-\alpha+1/2}$.
Here  $\varphi\in C_0^{\infty}(\R)$ is   an even function, identically equals $1$ on
$B(0,M)$  and supported on $B(0,2M)$, where $M=M(N)$ is large enough
such that
\begin{align}\label{e2.7a}
\inf_{i=1,2}\inf_{|r|\geq M}|a_{i}(r)| \geq c_{low}>0
\end{align}
and there exist $\{\theta_{i}\}_{i=1,2}$ such that
\begin{align}\label{e2.8a}
\sup_{i=1,2}\sup_{|r|\geq M}|\arg a_{i}(r)-\theta_{i}|\leq 10^{-2}.
\end{align}
 Then we can split the Fourier multiplier of the operator $\frak M^\alpha_1 $ into three parts as in \eqref{e2.4} above.

\begin{lemma}\label{lem2.1}
For $q\geq p\geq1$, we have
	\begin{eqnarray}\label{e2.5} \hspace{1cm}
		\big\|\sup_{t\in[1,2]} |\widehat{m_\alpha}(tD) \varphi(t|D|))f|\big \|_{L^q(\R^n)}+ \big\|\sup_{t\in[1,2]} |{\mathcal E}(t|D|)f |\big \|_{L^q(\R^n)} \leq C\|f\|_{L^p(\Rn)}.
	\end{eqnarray}
\end{lemma}

\begin{proof}
 For all $\tau\geq1$ and $t\in[1,2]$, we note that
  $\varphi(t|\xi|)\widehat{m_\alpha}(t\xi)\lb \xi\rb^{\tau}$ is smooth and compactly supported and ${\mathcal E}(t|\xi|)\lb \xi\rb^{\tau}\in {\mathscr S}(\Rn)$. Then by Young's inequality and Sobolev embedding, we have
\begin{align*}
&\big\|\sup_{t\in[1,2]} |\widehat{m_\alpha}(tD) \varphi(t|D|))f|\big\|_{L^q(\R^n)}+ \big\|\sup_{t\in[1,2]} |{\mathcal E}(t|D|)f |\big \|_{L^q(\R^n)}\\
&\leq C\|\lb \cdot\rb^{-n-1}\ast(\lb D\rb^{-\tau}f)\|_{L^q(\R^n)}\leq C\|\lb D\rb^{-\tau}f\|_{L^q(\R^n)}\\
&\leq C\|f\|_{L^p(\R^n)}
\end{align*}
if we choose $\tau> n(\frac{1}{p}-\frac{1}{q})$.
\end{proof}

Define
$$
{\mathscr E}_{N}f(x,t)=\int_{\Rn} e^{2\pi i(x\cdot \xi+t|\xi|)}{\mathcal E}_{N,1}\big(t|\xi|\big) \hat{f}(\xi)\,\text{d}\xi+\int_{\Rn} e^{2\pi i(x\cdot \xi-t|\xi|)}{\mathcal E}_{N,2}\big(t|\xi|\big)  \hat{f}(\xi)\,\text{d}\xi.
$$
Then   we have the following lemma.
\begin{lemma}\label{le2.2}
	Let $q\geq p\geq1$.  There exists a constant $C>0$ such that
	\begin{eqnarray}\label{e2.6}
		\big\|\sup_{t\in [1,2]}|{\mathscr E}_{N}f(\cdot,t)| \big\|_{L^q(\Rn)}\leq C  \|f\|_{L^p(\Rn)},
	\end{eqnarray}
	when
	$$
	N>\max\left\{2\left(\frac{1}{q}-{\rm Re}\,\alpha\right), 2n\left(\frac{1}{p}-\frac{1}{q}\right)\right\}.
	$$
\end{lemma}

\begin{proof}
We fix a function $\varphi$ as in \eqref{e2.4}. Let $\psi(r):=\varphi(r)-\varphi(2r)$ and $\psi_j(r):=\psi(2^{-j}r)$, for $j\geq 1$. So we have
\begin{align}\label{eee3.1}
  1\equiv  \varphi(r) +  \sum_{j\geq 1 }\psi_j(r), \quad r\geq0.
 \end{align}
For $j\geq1$, define
	$$
	{\mathscr E}_{N,j}f(x,t)=\int_{\mathbb R^n}  \left(e^{2\pi i(x\cdot \xi+t|\xi|)}{\mathcal E}_{N,1}\big(t|\xi|\big) + e^{2\pi i(x\cdot \xi-t|\xi|)}{\mathcal E}_{N,2}\big(t|\xi|\big)\right)  {\psi_j}(t|\xi|)\hat{f}(\xi)d\xi.
	$$
	To prove  \eqref{e2.6}, it suffices to show that there exists a constant  $\delta>0$ such that
	for all $j\geq 1$,
	\begin{eqnarray}\label{e2.7}
		\big\|\sup_{1\leq t\leq 2}|{\mathscr E}_{N,j}f(\cdot,t)| \big\|_{L^q(\Rn)}\leq C 2^{-\delta j}\|f\|_{L^p(\Rn)}.
	\end{eqnarray}

First, for each fixed $t\in[1,2]$, ${\mathscr E}_{N,j}\lb D\rb^{N/2}$ are the sum of two Fourier integral operators of order $-(n-1)/2-{\rm Re}\, \alpha-N/2$ with phase $x\cdot\xi\pm t|\xi|$.
	By \cite[Theorem 2, Chapter IX]{St1}  and the fact that $e^{it\sqrt{-\Delta}}$ is local at scale $t$, we have
	\begin{align}\label{e2.8}
	  \sup_{1\leq t\leq2}\big\|	 {\mathscr E}_{N,j}f(\cdot,t)\big\|_{L^q(\Rn)}&\leq C 2^{-((n-1)/2+{\rm Re}\,\alpha+N/2) j}2^{(n-1)|1/2-1/q|j}\|\lb D\rb^{-N/2}f\|_{L^q(\Rn)}\nonumber\\
&\leq C 2^{-({\rm Re}\,\alpha+N/2) j}\|f\|_{L^p(\Rn)},
	\end{align}
if we choose $N> 2n\left(\frac{1}{p}-\frac{1}{q}\right)$.

Next, we write 	 $\partial_{t}{\mathscr E}_{N,j}f(x,t)$ as the sum of following terms,
	\begin{eqnarray*}
		&&	\pm2\pi it^{-1}\int e^{2\pi i(x\cdot \xi\pm t|\xi|)}t|\xi|{\mathcal E}_{N,1}\big(t|\xi|\big) {\psi_j}(t|\xi|)\hat{f}(\xi)d\xi;\\
		&&
			 	\pm2\pi it^{-1}\int e^{2\pi i(x\cdot \xi\pm t|\xi|)}t|\xi|{\mathcal E}_{N,2}\big(t|\xi|\big) {\psi_j}(t|\xi|)\hat{f}(\xi)d\xi;\\
		&&
		t^{-1}\int e^{2\pi i(x\cdot \xi\pm t|\xi|)}t|\xi|({\mathcal E}_{N, 1}\psi_{j})'\big(t|\xi|\big) \hat{f}(\xi)d\xi;\\
		&&
		t^{-1}\int e^{2\pi i(x\cdot \xi\pm t|\xi|)}t|\xi|({\mathcal E}_{N, 2}\psi_{j})'\big(t|\xi|\big) \hat{f}(\xi)d\xi.
	\end{eqnarray*}
	By \eqref{e2.3}, we see that for each fixed $t\in[1,2]$, they are Fourier integral operators of order no more than $-(n-1)/2-{\rm Re}\, \alpha-N+1$.
	By \cite[Theorem 2, Chapter IX]{St1} again,
	\begin{align}\label{e2.9}
		\sup_{1\leq t\leq2}\big\|\partial_t{\mathscr E}_{N,j}f(\cdot,t) \big\|_{L^q(\Rn)}\leq C 2^{-({\rm Re}\,\alpha+N/2-1) j}\|f\|_{L^p(\Rn)}.
	\end{align}

With \eqref{e2.8} and  \eqref{e2.9} at our disposal,  we can apply    \cite[Lemma 2.4.2]{S1} to obtain
	$$
	\big\|\sup_{1\leq t\leq 2}|{\mathscr E}_{N,j}f(\cdot,t)| \big\|_{L^q(\Rn)}\leq C 2^{-({\rm Re}\,\alpha+N/2-1/q) j}\|f\|_{L^p(\Rn)}.
	$$
	Choosing $N>\max\{2(1/q-{\rm Re}\,\alpha),2n\left(\frac{1}{p}-\frac{1}{q}\right)\}$ and letting    $\delta={\rm Re}\,\alpha+N/2-1/q$, we  obtain  estimate \eqref{e2.7}. The  proof of Lemma~\ref{le2.2} is complete.
\end{proof}

\bigskip

At the end of this section, we define
\begin{align}\label{e2.10}
	{\mathscr A}_t f(x)= \int_{\Rn} \left( e^{2\pi i(x\cdot\xi+ t|\xi|)}a_1(t|\xi|) +  e^{2\pi i(x\cdot\xi- t|\xi|)}a_2(t|\xi|) \right) \hat{f}(\xi) \,\text{d}\xi.
\end{align}
From \eqref{e2.4}, Lemmas~\ref{lem2.1} and \ref{le2.2}, we see that the $(p,q)$-boundedness of
the operator $\frak M^\alpha $   reduces to the boundedness of the operator
${\mathscr A}_t $ on Sobolev spaces,
which will be investigated  in Section 3 below.

\medskip

\section{Proof of     Theorem~\ref{th1.1}}
\setcounter{equation}{0}


To prove    Theorem~\ref{th1.1}, we need to show the following Lemmas \ref{prop3.1}, \ref{prop3.4} and \ref{prop3.5}.

\begin{lemma}\label{prop3.1}
	Let $n\geq 2$ and $1\leq p,q\leq\infty$. Suppose
	\begin{align}\label{e3.1}
		\big\| {\frak M}_1^\alpha f\big\|_{L^q(\Rn)}\leq C\|f\|_{L^p(\Rn)}
	\end{align}
	holds for some $\alpha\in \mathbb{C}$. Then, we have $q\geq p$ and
	$${\rm Re}\,\alpha\geq \frac{1}{p}-\frac{n}{q}.
	$$
\end{lemma}

\begin{proof}
  Fix  $N>\max\left\{2\big(\frac{1}{q}-{\rm Re}\,\alpha\big), 2n\left(\frac{1}{p}-\frac{1}{q}\right)\right\}$.  Let ${\mathscr A}_1$ be an operator given in \eqref{e2.10}.
From \eqref{e2.4}, Lemma~\ref{lem2.1} and Lemma~\ref{le2.2}, we see that   the proof of Lemma~\ref{prop3.1} reduces to  the following result:
	Suppose
	\begin{align}\label{e3.2}
		\|{\mathscr A}_1f \|_{L^q(\Rn)}\leq C\|f\|_{W^{s,p}(\Rn)}
	\end{align}
	holds for some $s\in\R$. Then  we have $q\geq p$ and
$$s\geq \frac{n-1}{2}+\frac{1}{p}-\frac{n}{q}.$$

Indeed, since  ${\mathscr A}_1$ is
translation-invariant,  it follows by \cite[Proposition 2.5.6]{G1} that $q\geq p$.
Let $\phi\in C^{\infty}_{c}(\R\setminus\{0\})$ be nonnegative and equal to one  on $[1,2]$. For $j \geq1$, define
$$
\widehat{f_{j}}(\xi):=e^{-2\pi i|\xi|}\phi(2^{-j}|\xi|).
$$
It follows \cite[Lemma 2.1]{Z} that $\|f_{j}\|_{L^{p}(\R^{n})}
\leq C2^{(\frac{n+1}{2}-\frac{1}{p})j}.$  Then we have
\begin{align}\label{eee3.0}
\|{\mathscr A}_1f_{j}\|_{L^{q}(\R^{n})}\leq C\|f_{j}\|_{W^{s,p}(\R^{n})}\leq C 2^{sj}\|f_{j}\|_{L^{p}(\R^{n})}
\leq C2^{(s+\frac{n+1}{2}-\frac{1}{p})j}.
\end{align}
If $j$ large enough, we may apply  \eqref{e2.7a} and \eqref{e2.8a} to obtain that
\begin{align*}
\left|\int_{\Rn}a_{1}(|\xi|)\phi(2^{-j}|\xi|)\,d\xi\right|&\geq C^{-1}\int_{\Rn}\left|a_{1}(|\xi|)\phi(2^{-j}|\xi|)\right|\,d\xi\geq
C^{-1}c_{low}\int_{\Rn}\left|\phi(2^{-j}|\xi|)\right|\,d\xi\\
&\geq C^{-1}c_{low}2^{nj}.
\end{align*}
 Then  taking  $\varepsilon>0$ small enough, we use \eqref{eee3.1} to obtain that for all  $|x|\leq \varepsilon 2^{-j}$,
\begin{align}\label{eee3.2}
\left|e^{2\pi i\sqrt{-\Delta}}a_{1}(|D|)f_{j}(x)\right|&=\left|\int_{\Rn}a_{1}(|\xi|)\phi(2^{-j}|\xi|) e^{-2\pi ix\cdot \xi}\, d\xi\right|\nonumber\\
&=\left|\int_{\Rn}a_{1}(|\xi|)\phi(2^{-j}|\xi|)\,d\xi+\int_{\Rn}a_{1}(|\xi|)\phi(2^{-j}|\xi|) \big(e^{-2\pi ix\cdot \xi}-1\big)\,d\xi\right|\nonumber\\
&\geq\left|\int_{\Rn} a_{1}(|\xi|)\phi(2^{-j}|\xi|)d\xi\right|-C\varepsilon \int \left|a_{1}(|\xi|)\phi(2^{-j}|\xi|)\right|\,d\xi\nonumber\\
&\geq (2C)^{-1}\int \left|a_{1}(|\xi|)\phi(2^{-j}|\xi|)\right|\,d\xi\geq c2^{nj}
\end{align}
for some $c>0$. We can write
\begin{align*}
e^{-2\pi i\sqrt{-\Delta}}a_{2}(|D|)f_{j}(x)=\int_{\Rn}e^{-2\pi ix\cdot\xi}e^{-4\pi i|\xi|}a_{2}(|\xi|)\phi(2^{-j}|\xi|)\,d\xi.
\end{align*}
Note that the phase function $-2\pi x\cdot\xi-4\pi |\xi|$ has no critical points when $|x|\leq \varepsilon$ and $\varepsilon$ is small enough. So, by integration by parts, we have
\begin{align*}
\sup_{|x|\leq \varepsilon 2^{-j}}|e^{-2\pi i\sqrt{-\Delta}}a_{2}(|D|)f_{j}(x)|\leq C,
\end{align*}
which, combined with \eqref{eee3.2}, implies
\begin{align}\label{eee3.3}
\|{\mathscr A}_1f_{j}\|_{L^{q}(\R^{n})}\geq \|{\mathscr A}_1f_{j}\|_{L^{q}(|x|\leq \varepsilon 2^{-j})}
\geq C_{\varepsilon}2^{(n-n/q)j}.
\end{align}
Combining \eqref{eee3.0} and \eqref{eee3.3}, we have $C_{\varepsilon}2^{(n-n/q)j}\leq C2^{(s+\frac{n+1}{2}-\frac{1}{p})j}$. Letting $j\rightarrow+\infty$, we conclude that $s\geq\frac{n-1}{2}+\frac{1}{p}-\frac{n}{q}$.
\end{proof}

\begin{lemma}\label{prop3.4}
	Let $n\geq 2$, $1\leq p,q\leq\infty$. Suppose
	\begin{align}\label{e3.11}
		\big\| \sup_{1\leq t\leq 2}|\frak{M}^\alpha_t f| \big\|_{L^q(\Rn)}\leq C\|f\|_{L^p(\Rn)}
	\end{align}
	holds for some $\alpha\in \mathbb{C}$. Then, we have $q\geq p$ and
	$$
	{\rm Re}\,\alpha\geq \frac{n+1}{2p}-\frac{n-1}{2q}-\frac{n-1}{2}.
	$$
\end{lemma}

\begin{proof}
 Fix  $N>\max\left\{2\left(\frac{1}{q}-{\rm Re}\,\alpha\right), 2n\left(\frac{1}{p}-\frac{1}{q}\right)\right\}$ as  in Lemma~\ref{le2.2}.
From \eqref{e2.4}, Lemma~\ref{lem2.1} and   Lemma~\ref{le2.2}, we see that  the proof of Lemma~\ref{prop3.4}
reduces to show the following: Suppose
	\begin{align}\label{e3.4}
		\big\|\sup_{1\leq t\leq 2}|{\mathscr A}_t f|\big\|_{L^q(\R^n)}\leq C\|f\|_{W^{s,p}(\Rn)}
	\end{align}
	holds for some $s\in\R$. 	Then   we have $q\geq p$ and $s\geq\frac{n+1}{2p}-\frac{n-1}{2q}$.

Note that \eqref{e3.4} implies that $\big\|{\mathscr A}_t  f\big\|_{L^q(\R^n)}\leq C\|f\|_{W^{s,p}(\Rn)}$. This fact,  together with the translation-invariant property of ${\mathscr A}_t$,  yields that $q\geq p$.
Now let us prove $s\geq\frac{n+1}{2p}-\frac{n-1}{2q}$. Denote  $\xi= (\xi_1, \xi^\prime)\in \Rn$. For $j\geq1$ and $\delta>0$, we let $\hat{f}\geq0$  be a smooth cut-off function of
	the set
	\begin{align*}
		 \left\{(\xi_1,\xi^\prime)\in\Rn:|\xi_1-2^j|\leq \delta 2^{j-1}, |\xi^\prime|\leq \delta 2^{j/2}\right\}
	\end{align*}
	such that $\big|\partial_{\xi}^{\beta} \hat f(\xi)\big|\leq C_{\delta,\beta} 2^{-j|\beta'|/2}2^{-j|\beta_{1}|}$ for any $\beta=(\beta_{1},\beta')\in \mathbb{Z}_{+}^{n}$.

    It follows from (3.22) of \cite{LSSY} that, if $j$ is large enough and $\delta$ is small enough, we have
	\begin{align}\label{e3.18}
		\sup_{1\leq t\leq 2}|{\mathscr A}_t f|\geq C^{-1}\delta^{n}2^{\frac{n+1}{2}j},
	\end{align}
	for all $1\leq x_1\leq 2$, $|x^\prime|\leq 2^{-j/2}$. Then we have
\begin{align}\label{eee3.6}
\big\|\sup_{1\leq t\leq 2}|{\mathscr A}_t f|\big\|_{L^q(\R^n)}
\geq \big\|\sup_{1\leq t\leq 2}|{\mathscr A}_t f|\big\|_{L^{q}(1\leq x_1\leq 2,\, |x^\prime|\leq 2^{-j/2} )}
\geq C^{-1}\delta^{n}2^{\frac{n+1}{2}j}2^{-\frac{n-1}{2q}j}.
\end{align}
By \eqref{e3.4} and the definition of $f$, we have
\begin{align}\label{eee3.7}
\big\|\sup_{1\leq t\leq 2}|{\mathscr A}_t f|\big\|_{L^q(\R^n)}
\leq C\|f\|_{W^{s,p}} \leq C_\delta 2^{sj}2^{\frac{n+1}{2}j-\frac{n+1}{2p}j}.
\end{align}
Combining \eqref{eee3.6} and \eqref{eee3.7}, we obtain
\begin{align*}
C^{-1}\delta^{n}2^{\frac{n+1}{2}j}2^{-\frac{n-1}{2q}j}\leq C_\delta 2^{sj}2^{\frac{n+1}{2}j-\frac{n+1}{2p}j}.
\end{align*}
Let $j\rightarrow+\infty$ and we have $\frac{n+1}{2}-\frac{n-1}{2q}\leq s+\frac{n+1}{2}-\frac{n+1}{2p}$, which means $s\geq\frac{n+1}{2p}-\frac{n-1}{2q}$.
\end{proof}

\begin{lemma}\label{prop3.5}
	Let $n\geq 2$, $1\leq p,q\leq\infty$. Suppose
	\begin{align}\label{ees3.11}
		\big\| \sup_{1\leq t\leq 2}|\frak{M}^\alpha_t f| \big\|_{L^q(\Rn)}\leq C\|f\|_{L^p(\Rn)}
	\end{align}
	holds for some $\alpha\in \mathbb{C}$. Then, we have $q\geq p$ and
	$$
	{\rm Re}\,\alpha\geq \frac{n}{p}-n+1.
	$$
\end{lemma}

\begin{proof}
Fix  $N>\max\left\{2\left(\frac{1}{q}-{\rm Re}\,\alpha\right), 2n\left(\frac{1}{p}-\frac{1}{q}\right)\right\}$ as  in Lemma~\ref{le2.2}.
By using \eqref{e2.4} and Lemmas~\ref{lem2.1} and \ref{le2.2},   Lemma \ref{prop3.5}
reduces to  the following result:   Suppose
	\begin{align}\label{e3.4s}
		\big\|\sup_{1\leq t\leq 2}|{\mathscr A}_t f|\big\|_{L^q(\R^n)}\leq C\|f\|_{W^{s,p}(\Rn)}
	\end{align}
	holds for some $s\in\R$. 	Then   we have $q\geq p$ and $s\geq\frac{n}{p}-\frac{n-1}{2}$.

Since ${\mathscr A}_t$ is translation-invariant, we have that $q\geq p$.
Now let us prove $s\geq\frac{n}{p}-\frac{n-1}{2}$.  Assume that $\chi(\xi)\in C^\infty(\Rn\backslash \{0\})$ is homogeneous of order $0$ satisfying
$\chi(\xi)=1$ if $|{\xi\over |\xi|}-v_1|\leq 10^{-2}$
and vanishes if $|{\xi\over |\xi|}-v_1|\geq 9^{-2}$, where $v_1:=(1,0,\cdots,0)$. Let $\hat{f_{j}}(\xi):=\phi(2^{-j}|\xi|)\chi(\xi)$, where $\phi\in C^{\infty}_{c}(\R\setminus\{0\})$, $\phi=1$ on $[1,2]$ and $\phi\geq0$.

Note that
\begin{align*}
e^{2\pi it\sqrt{-\Delta}}a_{1}(t|D|)f_{j}(x)&=\int e^{2\pi ix\cdot\xi}e^{2\pi it|\xi|}a_{1}(t|\xi|)\phi(2^{-j}|\xi|)\chi(\xi)\,d\xi.
\end{align*}
If $|{x\over |x|}-v_1|\leq 10^{-2}$, the phase function of the above integral has no critical points and thus
\begin{align}\label{eee3.10}
\sup\limits_{t\in [1,2]}|e^{2\pi it\sqrt{-\Delta}}a_{1}(t|D|)f_{j}(x)|\leq C2^{-nj}.
\end{align}

It is known from \cite[p. 360]{St1} that for $|x|\geq 1$ and $|{x\over |x|}-v_1|\leq 9^{-2}$, there holds
	\begin{align*}
		 \widehat{\chi\text{d}\sigma}(-x)= e^{2\pi i|x|}h(-x)+e(-x),
	\end{align*}
	where $e$ belongs to ${\it S}^{-\infty}$ and $h\in {\it S}^{-(n-1)/2}$ can be splited into two terms:
	\begin{eqnarray*}
		h(x)= c_0|x|^{-(n-1)/2}\chi(-x/|x|)+ \tilde{e}(x),\ \ \tilde{e}\in {\it S}^{-(n+1)/2}
	\end{eqnarray*}
for all $|x|\geq1$.
Then for all $|x|\geq1$, $|{x\over |x|}-v_1|\leq 9^{-2}$ and $j\geq1$, we have
\begin{align*}
&e^{-2\pi it\sqrt{-\Delta}}a_{2}(t|D|)f_{j}(x)\\
&=\int e^{2\pi ix\cdot\xi}e^{-2\pi it|\xi|}a_{2}(t|\xi|)\phi(2^{-j}|\xi|)\chi(\xi)\,d\xi\\
&=\int_0^\infty (\chi d\sigma)^\wedge(-rx)e^{-2\pi itr}a_{2}(tr)\phi(2^{-j}r)r^{n-1}\,dr\\
&=c_{0}|x|^{-\frac{n-1}{2}}\int_{0}^{\infty}  e^{2\pi ir(|x|-t)}\chi(x/|x|)a_{2}(tr)\phi(2^{-j}r) r^{\frac{n-1}{2}} \,\text{d}r+R_{j}(x,t),
\end{align*}
where
\begin{align*}
R_{j}(x,t):=&c_{1}\int e^{2\pi ir(|x|-t)}\tilde{e}(-rx)a_{2}(tr)\phi(2^{-j}r)r^{n-1}\,dr\\
&+c_{2}\int e^{-2\pi itr}e(-rx)a_{2}(tr)\phi(2^{-j}r)r^{n-1}\,dr
\end{align*}
and $c_{0},c_{1},c_{2}>0$.
Then if $|{x\over |x|}-v_1|\leq 10^{-2}$ and $\big||x|-t\big|\leq \delta 2^{-j}$ for some $\delta\in(0,1)$ small enough, by \eqref{e2.7a} and \eqref{e2.8a} we have
\begin{align*}
\big|e^{-2\pi it\sqrt{-\Delta}}a_{2}(t|D|)f_{j}(x)\big|\geq c\bigg|\int \phi(2^{-j}r)r^{\frac{n-1}{2}}dr\bigg|-C2^{\frac{n-1}{2}j}\geq \tilde{c} 2^{\frac{n+1}{2}j}
\end{align*}
for some positive constant $\tilde{c}$, if $j$ large enough. Then if $|x|\in[1,2]$ and $\big|{x\over |x|}-v_1\big|\leq 10^{-2}$,
$$
\sup_{t\in[1,2]}\big|e^{-2\pi it\sqrt{-\Delta}}a_{2}(t|D|)f_{j}(x)\big|\geq \tilde{c} 2^{\frac{n+1}{2}j},
$$
which, combined with \eqref{eee3.10}, yields

\begin{align*}
\sup_{t\in[1,2]}|{\mathscr A}_t f_{j}| \geq \sup_{t\in[1,2]}\big|e^{-2\pi it\sqrt{-\Delta}}a_{2}(t|D|)f_{j}\big|-\sup_{t\in[1,2]}\big|e^{2\pi it\sqrt{-\Delta}}a_{1}(t|D|)f_{j}\big|
\geq \frac{\tilde{c}}{2} 2^{\frac{n+1}{2}j}
\end{align*}
if $j$ large enough.
Hence,
\begin{align}\label{eee3.11}
\big\|\sup_{t\in[1,2]}|{\mathscr A}_tf_{j}|\big\|_{L^{q}(\Rn)}\geq \big\|\sup_{t\in[1,2]}|{\mathscr A}_tf_{j}|\big\|_{L^{q}(|x|\in[1,2], |{x\over |x|}-v_1|\leq 10^{-2})}
\geq  C^{-1}\frac{\tilde{c}}{2} 2^{\frac{n+1}{2}j}.
\end{align}
On the other hand, choose $\tilde{\phi}\in C^{\infty}_{c}(\R\setminus\{0\})$ satisfying $\tilde{\phi}=1$ on $\supp \phi$, then $\hat{f_{j}}(\xi)=\phi(2^{-j}|\xi|)\tilde{\phi}(2^{-j}|\xi|)\chi(\xi)$. Since $\tilde{\phi}(2^{-j}|D|)\chi(D)$ is bounded on $L^{p}(\Rn)$, we have
\begin{align}\label{eee3.12}
\big\|\sup_{t\in[1,2]}|{\mathscr A}_tf_{j}|\big\|_{L^{q}(\Rn)}
\leq C\|f_{j}\|_{W^{s,p}(\Rn)}\leq C2^{sj}2^{nj}2^{-\frac{n}{p}j}.
\end{align}
By \eqref{eee3.11} and \eqref{eee3.12}, we conclude
\begin{align*}
C^{-1}\frac{\tilde{c}}{2}2^{\frac{n+1}{2}j}
\leq C2^{sj}2^{nj}2^{-\frac{n}{p}j}.
\end{align*}
Letting $j\rightarrow +\infty$,  we obtain $s\geq\frac{n}{p}-\frac{n-1}{2}$.
\end{proof}

We finally present the endgame in the

\begin{proof}[Proof of   Theorem~\ref{th1.1}]
This is a consequence of Lemmas \ref{prop3.1}, \ref{prop3.4} and \ref{prop3.5}.
\end{proof}

\bigskip

\section{Proof of Theorem~\ref{th1.2}}
\setcounter{equation}{0}


To prove Theorem \ref{th1.2},  we will first give
the $L^p\to L^q$ local smoothing estimates for the wave operator  $e^{it\sqrt{-\Delta}}$, i.e.,     Theorem \ref{thm4.3} for $n=2$   and  Theorem \ref{thm4.4} for  $n>2$ below.

\begin{theorem}\label{thm4.3}
Let $1\leq p\leq q\leq\infty$. Denote $s_2(p,q)$  as follows.
\begin{equation}
s_2(p,q):=
\begin{cases}
\frac{1}{2}+\frac{1}{p}-\frac{3}{q}, &\text{for } q\geq3p', \\
\frac{3}{2p}-\frac{3}{2q}, &\text{for } p'\leq q<3p', \\
\frac{2}{p}-\frac{1}{2}-\frac{1}{q}, &\text{for } q<p',
\end{cases}
\end{equation}
where $p'=p/(p-1)$. If $s>s_2(p,q)$, then we have
\begin{align}\label{e4.2a}
\bigg(\int_{1}^{2}\|e^{it\sqrt{-\Delta}}f\|_{L^{q}(\R^{2})}^{q}dt\bigg)^{1/q}\leq C_{\varepsilon}\|f\|_{W^{s,p}(\R^{2})}.
\end{align}
\end{theorem}

\begin{proof}
Let $\varepsilon$ be any positive number. By the local smoothing estimate of Guth,Wang and Zhang  \cite{GWZ}, we have the following $(4,4)$ estimate:
\begin{align}\label{ee4.1}
\bigg(\int_{1}^{2}\|e^{it\sqrt{-\Delta}}f\|_{L^{4}(\R^{2})}^4\,dt\bigg)^{1/4}\leq C_{\varepsilon}\|f\|_{W^{\varepsilon,4}(\R^{2})}.
\end{align}
It is also known  that the fixed-time estimate of Seeger, Sogge and Stein \cite{SSS} implies the following $(1,1)$ estimate and $(\infty,\infty)$ estimate:
\begin{align}\label{ee4.2}
\|e^{it\sqrt{-\Delta}}f\|_{L^{p}(\R^{2}\times[1,2])}\leq C_{\varepsilon}\|f\|_{W^{\frac{1}{2}+\varepsilon,p}(\R^{2})},
\end{align}
where $p=1$ or $\infty$. Moreover, it follows from \cite[Chapter IX, 6.16]{St1} that  the following $(1,\infty)$ estimate holds
\begin{align}\label{ee4.3}
\sup_{t\in[1,2]}\|e^{it\sqrt{-\Delta}}f\|_{L^\infty(\R^{2})}\leq C\|f\|_{W^{\frac{3}{2}+\varepsilon,1}(\R^{2})}.
\end{align}

 Theorem \ref{thm4.3} can be proved by interpolation and the estimates \eqref{ee4.1}--\eqref{ee4.3}. More precisely, in the case $q\geq3p'$,  \eqref{e4.2a} follows from the interpolation between $(\infty,\infty)$, $(4,4)$ and  $(1,\infty)$ estimates. In the case $p'<q\leq 3p'$,  \eqref{e4.2a} follows from the interpolation between $(4,4)$, $(2,2)$ and  $(1,\infty)$ estimates.  In the case $q<p'$, \eqref{e4.2a} follows from the interpolation between $(2,2)$, $(1,1)$ and  $(1,\infty)$ estimates.
\end{proof}

The following estimate is proven in \cite[Proposition 2.1]{CLL}.
\begin{lemma}\label{lemma4.2}
Let $n\geq3$ and suppose $p_{0}=\frac{2(n^{2}+2n-1)}{(n-1)(n+3)}$, $q_{0}=\frac{2(n^{2}+2n-1)}{(n-1)(n+1)}$, $s_{0}=\frac{(n-1)(n+1)}{2(n^{2}+2n-1)}$. Then
\begin{align*}
\bigg(\int_{1}^{2}\|e^{it\sqrt{-\Delta}}f\|_{L^{q_{0}}(\R^{n})}^{q_{0}}dt\bigg)^{1/q_{0}}\leq C2^{s_{0}k}\|f\|_{L^{p_{0}}(\R^{n})}
\end{align*}
holds for all $k\geq1$ and $f\in \mathcal{S}'$ with $\supp \hat{f}\subseteq \{\xi\in\Rn: \  2^{k-1}\leq |\xi|\leq 2^{k+1}\}$.
\end{lemma}

Recall that $O, A, B, C, D, E$ are given in Figure 1 in   Section 1.

\begin{theorem}\label{thm4.4}
Suppose $n\geq3$. Let $1\leq p\leq q\leq\infty$. Denote $s_{n}(p,q)$  as follows.
\begin{equation*}
s_{n}(p,q):=
\begin{cases}
\frac{1}{p}-\frac{n+1}{q}+\frac{n-1}{2}, &\text{for } (\frac{1}{p},\frac{1}{q})\in \Delta AOE\cup\Delta ABE,\\
\frac{n+1}{2}(\frac{1}{p}-\frac{1}{q}), &\text{for } (\frac{1}{p},\frac{1}{q})\in \Delta BCE,\\
\frac{1}{2p}-\frac{n}{2q}+\frac{n-1}{4}, &\text{for } (\frac{1}{p},\frac{1}{q})\in \Delta ABC,\\
\frac{n}{p}-\frac{1}{q}-\frac{n-1}{2}, &\text{for } (\frac{1}{p},\frac{1}{q})\in \Delta CDE.
\end{cases}
\end{equation*}
If $s>s_{n}(p,q)$, we have
\begin{align}\label{e4.5a}
\bigg(\int_{1}^{2}\|e^{it\sqrt{-\Delta}}f\|_{L^{q}(\R^{n})}^{q}dt\bigg)^{1/q}\leq C_{s}\|f\|_{W^{s,p}(\R^{n})}.
\end{align}
\end{theorem}

\begin{proof}
Let $\varepsilon$ be any positive number. By applying the decoupling estimate of Bourgain and Demeter \cite{BD}, we have that  for all $p\geq \frac{2n+2}{n-1}$,
\begin{align*}
\bigg(\int_{1}^{2}\|e^{it\sqrt{-\Delta}}f\|_{L^{p}(\R^{n})}\bigg)^{1/p}\leq C_{\varepsilon}\|f\|_{W^{\frac{n-1}{2}-\frac{n}{p}+\varepsilon,p}(\R^{n})}.
\end{align*}
It follows from the fixed-time estimate of Seeger-Sogge-Stein \cite{SSS}  that
\begin{align*}
\|e^{it\sqrt{-\Delta}}f\|_{L^{p}(\R^{n}\times[1,2])}\leq C_{\varepsilon}\|f\|_{W^{\frac{n-1}{2}+\varepsilon,p}(\R^{n})},
\end{align*}
where $p=1$ or $\infty$. By \cite[Chapter IX, 6.16]{St1}, we have the following $(1,\infty)$ estimate:
\begin{align*}
\sup_{t\in[1,2]}\|e^{it\sqrt{-\Delta}}f\|_{L^\infty(\Rn)}\leq C\|f\|_{W^{\frac{n+1}{2}+\varepsilon,1}(\Rn)}.
\end{align*}
By Lemma \ref{lemma4.2}, it is not difficult to obtain the following $(p_{0},q_{0})$ estimate:
\begin{align*}
\bigg(\int_{1}^{2}\|e^{it\sqrt{-\Delta}}f\|_{L^{q_{0}}(\R^{n})}^{q_{0}}dt\bigg)^{1/q_{0}}\leq C\|f\|_{W^{s_0+\varepsilon, p_0}(\R^{n})}.
\end{align*}

With these estimates at our disposal, Theorem \ref{thm4.4} can be proved by interpolation.  More precisely,  in the case $(\frac{1}{p},\frac{1}{q})\in \Delta AOE\cup\Delta ABE$,
 \eqref{e4.5a} follows from the interpolation between $(\infty,\infty)$, $(\frac{2n+2}{n-1},\frac{2n+2}{n-1})$, $(p_{0},q_{0})$ and  $(1,\infty)$ estimates.  In the case $(\frac{1}{p},\frac{1}{q})\in \Delta BCE$, \eqref{e4.5a} follows from the interpolation between $(p_{0},q_{0})$, $(2,2)$ and  $(1,\infty)$ estimates. In the case $(\frac{1}{p},\frac{1}{q})\in \Delta ABC$, \eqref{e4.5a} follows from the interpolation between $(2,2)$, $(p_{0},q_{0})$ and  $(\frac{2n+2}{n-1},\frac{2n+2}{n-1})$ estimates. In the case $(\frac{1}{p},\frac{1}{q})\in \Delta CDE$, \eqref{e4.5a} follows from the interpolation between $(2,2)$, $(1,1)$ and  $(1,\infty)$ estimates.
\end{proof}

\begin{lemma}\label{lem4.5}
For all $1\leq q\leq\infty$, we have
\begin{align*}
\big\|\sup_{t\in[1,2]}|g(\cdot,t)|\big\|_{L^{q}(\Rn)}\leq  C\|g\|_{L^{q}(\Rn\times[1/2,2])}+ C \|g\|_{L^{q}(\Rn\times[1/2,2])}^{1-\frac{1}{q}}\|\partial_{t}g\|_{L^{q}(\Rn\times[1/2,2])}^{\frac{1}{q}}.
\end{align*}
\end{lemma}

\begin{proof}
Choose $\rho\in C^{\infty}_{c}(1/2,4)$ satisfy $\rho(t)=1$ if $t\in[1,2]$. Then by \cite[Lemma 2.4.2]{S1}, we have
\begin{align*}
\big\|\sup_{t\in[1,2]}|g(x,t)|\big\|_{L^{q}(\Rn)}&\leq \big\|\sup_{t\in[1/2,2]}|\rho(t)g(x,t)|\big\|_{L^{q}(\Rn)}\\
&\leq  C \big\|\rho(t)g(x,t)\big\|_{L^{q}(\Rn\times[1/2,2])}^{1-\frac{1}{q}}\big\|\partial_{t}(\rho(t)g(x,t))\big\|_{L^{q}(\Rn\times[1/2,2])}^{\frac{1}{q}}\\
&\leq    C \big\|g(x,t)\big\|_{L^{q}(\Rn\times[1/2,2])}+ \big\|g(x,t)\big\|_{L^{q}(\Rn\times[1/2,2])}^{1-\frac{1}{q}}\big\|\partial_{t}g(x,t)\big\|_{L^{q}(\Rn\times[1/2,2])}^{\frac{1}{q}}.
\end{align*}
\end{proof}

\begin{proposition}\label{prop4.5}
	Let $n\geq 2 $ and  $1\leq p\leq q\leq\infty$.   If the local smoothing estimate
	\begin{align}\label{e4.1}
		\big\| e^{it\sqrt{-\Delta}}f \big\|_{L^q(\Rn\times [1,2])} \leq C_{n,p,q} \|f\|_{W^{s, p}(\Rn)}
	\end{align}
	holds for some   $s\in\R$,   then   we have
	\begin{align}\label{e4.2}
		\big\|\sup_{t\in[1,2]}|  {\frak M}_t^{\alpha}f |\big\|_{L^q(\Rn)}\leq C_{n,p,q,\alpha} \|f\|_{L^p(\Rn)}
	\end{align}
whenever
	${\rm Re}\, \alpha > s -\frac{n-1}{2} +\frac{1}{q}$.
\end{proposition}

\begin{proof}
Let  $\varphi$ and $\{\psi_j\}_{j\geq1}$ be functions  in \eqref{eee3.1}. We write
\begin{align}\label{e4.3}
	\widehat{{\frak M}_t^\alpha f}(\xi)
	&= \varphi(t|\xi|)\widehat{m_{\alpha}}(t\xi)\hat{f}(\xi)+\sum_{j\ge 1}\psi_j(t|\xi|)\widehat{m_{\alpha}}(t\xi)\hat{f}(\xi)              \notag \\
	&=:\widehat{{\frak M}_{0,t}^\alpha f}(\xi)+\sum_{j\ge 1}\widehat{{\frak M}_{j,t}^\alpha f}(\xi).
\end{align}
	By applying \eqref{e4.3} and Lemma \ref{lem2.1}, in order to prove \eqref{e4.2}, it suffices to  prove that for some $\delta>0$, there holds
	\begin{align}\label{e4.16}
		\big\|\sup_{t\in[1,2]}|  {\frak M}_{j,t}^{\alpha}f |\big\|_{L^q(\Rn)}\leq C2^{-\delta j} \|f\|_{L^p(\Rn)}
	\end{align}
	whenever
	${\rm Re}\, \alpha > s -\frac{n-1}{2} +\frac{1}{q}$.

	By using \eqref{e2.4},   together with  Lemmas \ref{lem2.1} and \ref{le2.2},   we reduce \eqref{e4.16}  to
	\begin{align}\label{e4.5}
		\big\|\sup_{t\in [1,2]}|  \mathscr{A}_{j,t}f |\big\|_{L^q(\Rn)}\leq C2^{(s+\frac{1}{q})j} \|f\|_{L^p(\Rn)},
	\end{align}
	where $\widehat{\mathscr{A}_{j,t}f}(\xi)=\psi_j(t|\xi|)\widehat{\mathscr{A}_{t}f}(\xi)$ and $\mathscr{A}_t f$ is defined in \eqref{e2.10}. By \eqref{ee2.1}, we can write
	\begin{align*}
		{\mathscr A}_{j,t} f(x)= C\sum_{\ell=0}^{N-1}\int_{\Rn} \left( b_{\ell} e^{2\pi i(x\cdot\xi+ t|\xi|)} +  d_{\ell} e^{2\pi i(x\cdot\xi- t|\xi|)}  \right)|t\xi|^{-\ell}\psi_j(t|\xi|) \hat{f}(\xi) \,\text{d}\xi,
	\end{align*}
	which is a linear combination of
	$$
	T_{\ell,j}f(x,t):=\int_{\Rn}e^{2\pi i(x\cdot\xi\pm t|\xi|)}|t\xi|^{-\ell}\psi_j(t|\xi|)\hat{f}(\xi)\,\text{d}\xi,\ \ \ \ell=0,1,\cdots,N-1.
	$$
	Hence, the proof of  \eqref{e4.5}   reduces to  showing that
	\begin{align}\label{e4.11}
		\big\|\sup_{t\in [1,2]}|  T_{0,j}f(\cdot,t) |\big\|_{L^q(\Rn)}\leq C2^{(s+\frac{1}{q})j}  \|f\|_{L^p(\Rn)}, \ \ \ j\geq 1.
	\end{align}
	
	We observe that for any $1\leq t\leq 2$ and $j\ge 1$, there holds
	$$
	\big|\partial_\xi^\beta \big(\psi_j(t|\xi|)\big)\big|\leq C(1+|\xi|)^{-|\beta|},
	$$
	where $\beta$ is any multi-index. So $\psi_j(t|\cdot|)\in S^0$ uniformly $1\leq t\leq 2$ and $j\ge 1$, hence
	\begin{align}\label{e4.7}	
		\int_{\Rn}\bigg| \int_{\Rn} e^{2\pi i(x\cdot\xi\pm t|\xi|)}\psi_j(t|\xi|)\hat{f}(\xi)\,\text{d}\xi \bigg|^p\,\text{d}x\leq C	 \int_{\Rn}\bigg| \int_{\Rn} e^{2\pi i(x\cdot\xi\pm t|\xi|)}\tilde{\psi}_j(\xi)\hat{f}(\xi)\,\text{d}\xi \bigg|^p\,\text{d}x,
	\end{align}
	where constant $C$ is independent of $t$ and $j$. Here  $\tilde{\psi}_j$ equals to 1 if $|\xi|\in [2^{j-2}M,2^{j+1}M]$ and vanishes if $|\xi|\notin [2^{j-3}M,2^{j+2}M]$, so that $\tilde{\psi}_j$ equals to 1  on the support of $\psi_j(t|\cdot|)$ when $1\leq t\leq 2$. By applying the assumption \eqref{e4.1}  to  \eqref{e4.7}, we have
	\begin{align}\label{eee4.1}
		\|  T_{0,j}f \|_{L^q(\Rn\times [1/2,2])}\leq C 2^{s j}\|f\|_{L^p(\Rn)}.
	\end{align}
	By the same token, the operator
	$$
	\partial_t T_{0,j}f(x,t) = \int_{\Rn}e^{2\pi i(x\cdot\xi\pm t|\xi|)}\big(\pm 2\pi i|\xi|\psi_j(t|\xi|)+|\xi|(\psi_j)^\prime(t|\xi|)\big)\hat{f}(\xi)\,\text{d}\xi.
	$$
	satisfies
	\begin{align}\label{eee4.2}
		\|  \partial_tT_{0,j}f \|_{L^q(\Rn\times [1/2,2])}\leq C 2^{(s+1)j}\|f\|_{L^p(\Rn)}.
	\end{align}
With \eqref{eee4.1} and \eqref{eee4.2} at our disposal,
 \eqref{e4.11} follows from  Lemma \ref{lem4.5}.
\end{proof}

\bigskip

\begin{proof}[Proof of Theorem~\ref{th1.2}]  By noting that $\sigma_2(p,q)=s_2(p,q)-\frac{1}{2}+\frac{1}{q}$  and $d_n(p,q)=s_n(p,q)-\frac{n-1}{2}+\frac{1}{q}$ for $n>2$,  Theorem~\ref{th1.2}  is a direct consequence of Proposition \ref{prop4.5} and
 Theorems \ref{thm4.3}, \ref{thm4.4}.
\end{proof}

\medskip

\noindent
\begin{remark}
In  the dimension $n\geq 3$    Gao {\it et al.} \cite{GLMX} obtained improved local smoothing estimates for the wave equation, that is,
\begin{align*}
		\big\| e^{it\sqrt{-\Delta}}f \big\|_{L^p(\Rn\times [1,2])} \leq C_{n,p} \|f\|_{W^{s, p}(\Rn)}
	\end{align*}
holds with $s= (n-1)({1/2}-1/p)-\sigma$ for all $\sigma<2/p-1/2$  when
\begin{eqnarray*}
	p>
	\left\{
	\begin{array}{lll}
		{2(3n+5)\over 3n+1}, \ \ \  {\rm for}& n \ {\rm odd};  \\[4pt]
		{2(3n+6)\over 3n+2}, \ \ \  {\rm for}& n \ {\rm even}.
	\end{array}
	\right.
\end{eqnarray*}
Applying  Proposition~\ref{prop4.5}, we can improve (ii) of Theorem \ref{th1.2}.
However, the range of $\alpha$ is not optimal.
What happens when $n\geq 3$ remains open.
\end{remark}

\smallskip

\noindent
{\bf Acknowledgments.}
The authors   were supported  by National Key R$\&$D Program of China 2022YFA1005700.
 N.J. Liu was supported by   China Postdoctoral Science Foundation (No. 2024M763732).  L. Song was  supported by  NNSF of China (No. 12471097).


\medskip

\bibliographystyle{plain}

\end{document}